\documentclass[11pt]{article}
\usepackage{hyperref}
\usepackage[ansinew]{inputenc}
\usepackage{amssymb}
\usepackage{amsmath}
\usepackage[english]{babel}
\usepackage{cases}
\usepackage{setspace}
\newtheorem{thm}{Theorem}[section]
\newtheorem{cor}{Corollary}
\newtheorem{lem}{Lemma}[section]
\newtheorem{prop}{Proposition}[section]
\newtheorem{defn}{Definition}[section]
\newtheorem{rem}{Remark}[section]
\newtheorem{Example}{Example}
\newenvironment{proof}[1][Proof]{\textbf{#1.} }{\ \rule{0.5em}{0.5em}}
\newcommand\tgo{_{t \geq 0}}

\def\E{\mathbb{E}}
\def\F{\mathcal{F}}

\def\P{\mathbb{P}}

\hypersetup{
backref=true, 
pagebackref=true,
hyperindex=true, 
colorlinks=true, 
breaklinks=true, 
urlcolor= blue, 
linktoc=section
linkcolor= blue, 
bookmarks=true, 
}
\begin{document}
\title{On Semimartingale local time inequalities and Applications in SDE's }
\date{15 March 2012}
\maketitle

\author{\begin{center} {M. Benabdallah, S. Bouhadou, Y. OUKNINE }\footnote{Corresponding authors \\Email address:{\bf bmohsine@gmail.com}\\
Department of mathematics, Ibn Tofail University, Kenitra, B.P. 133, Morocco\\
 Email addresses: {\bf ouknine@ucam.ac.ma} (Y. Ouknine), {\bf sihambouhadou@gmail.com} \\ LIBMA Laboratory, Department of Mathematics, Faculty of Sciences Semlalia, Cadi Ayyad University, P.B.O. 2390 Marrakesh, Morocco.\\}$^{,}$ \footnote{This work is supported by Hassan II Academy of Sciences and Technology.\\
Email address:}
\end{center}}

\bigskip
\begin{spacing}{1.1}
\begin{abstract}
Using the balayage formula, we prove an inequality between the measures associated to local times of semimartingales. Our result extends the "comparison theorem of local times" of Ouknine $(1988)$, which is useful in the study of stochastic differential equations. The inequality presented in this paper covers the discontinuous case.
Moreover, we study the pathwise uniqueness of some stochastic differential equations involving local time of unknown process.
\end{abstract}

\bigskip


\newpage

\tableofcontents

\section{Introduction}
In contrast of ordinary differential equation, the theory of stochastic differential equations has two distinguished  notions of solutions and two uniqueness properties: Pathwise uniqueness (P.U) and uniqueness in law (U.L).
Roughly, the pathwise uniqueness asserts that two solutions on the same probability space with the same stochastic inputs agree almost surely while weak uniqueness asserts that two solutions agree in distributions (Precise definitions will be given later).\\
On his seminal work, S. Nakao \cite{Nakao}  has proved a pathwise uniqueness property for SDE with non regular datas. To do this he used a key lemma on martingale and bounded variation processes.
(See an elegant proof of the result in N. V. Krylov and A. K. Zvonkin \cite{zvonkin}).\\
In many references studying the existence and uniqueness of strong solutions of the Itô equations, the common idea is the construction of weak solutions together with the subsequent use of the celebrated pathwise uniqueness argument, obtained by Yamada and Watanabe \cite{yamada} which can be formulated as:
$$\mbox{weak existence}+\mbox{pathwise uniqueness}\Rightarrow \mbox{strong uniqueness}.$$
For this reason, many investigations were devoted to the problem of (P.U) of solutions of SDE's. \\
In this paper we are interested by the stochastic differential equations of the form:
\begin{eqnarray}\label{eqna}
X_{t}&=&X_{0}+ \int_{0}^{t}\sigma_{s}(X_{.},B_{.})\;dB_{s}+\int_{\mathbb{R}}L_{t}^{a}(X)\; \nu(da)
\end{eqnarray}
where the mapping function $\sigma:\mathbb{R}^{+}\times C(\mathbb{R}^{+},\mathbb{R})\times C(\mathbb{R}^{+},\mathbb{R})\rightarrow \mathbb{R}$ is measurable and adapted to the filtration $(\mathcal{C}_{t})_{t \geq 0}$, $L_{t}^{a}(X)$ stands for local time of a continuous semimartingale $X$ at $a$, $\nu $ is $\sigma$-finite measure on $\mathbb{R}$. This equations was studied by H. J. Engelbert \cite{ENgelbert-Yamada}, the particular case where $(\sigma_{t})_{t \geq 0}$ is given by a borel function $\sigma$ defining
$$\sigma_{t}(x,w)=\sigma(x_{t})\;\;\; \mbox{for all} \; (x,w) \in C(\mathbb{R}^{+},\mathbb{R}^{2})$$
was considered by several authors, c.f., for example, Le Gall \cite{Legall}, H. J. Engelbert and W. Schmidt \cite{Schmidt} and recently by R. Belfadli and Y. Ouknine \cite{belfa}.
Using the local time technique and the inequality between local times of continuous semimartingales given by Y. Ouknine \cite{Ouk Nakao}, we prove the pathwise uniqueness for the SDE (\ref{eqna}). In fact, this equality which was extended to discontinuous case by F. Coquet and Y. Ouknine \cite{coquet}  is very useful in the study of SDE, several works was done in litterature, e.g. Y. Ouknine \cite{Ouk Nakao}, Belfadli and Ouknine \cite{belfa}, Rutkowski \cite{M.Rutkowski}.
So, in the first part of this paper, we prove a general result for comparison theorem for local times. Roughly speaking, if $X$ and $Y$ are two semimartingales sharing the same set of zero, and if $X^{+}\leq Y^{+}$, then the measure $dL^{0}(X)$ is absolutely continuous with respect to the measure $dL^{0}(Y)$. \\Our proof is based on the use of balayage formula of Azéma and Yor \cite{Yor}, which is the key of the first order calculus. And which proves an efficient tool to obtain the comparison theorem for local times in its strong form.\\

The structure of this paper is as follows:\\
\underline{Section}(\ref{continuous cas}): We start it with the continuous version of the balayage formula and show
how to deduce from it our new comparison theorem for local times of continuous semimartingales.\\
\underline{Section} (\ref{cad case}):  Illustrate our main result
for the càdlàg semimartingales. \\
\underline{Section} (\ref{Application to SDE}) : It contains three subsections concentrating to
on the study of the pathwise uniqueness of some classes of SDE's.\\
\underline{Subsection} (\ref{SDE local}): We give necessary and sufficient condition for pathwise uniqueness of solutions for SDE's with local time component:
\begin{eqnarray}\label{Local}
dY_{t}&=&\sigma(Y_{t})\;dB_{t}+b(Y_{t})\;dt+\frac{1}{2}dL_{t}^{0}(Y),\;\;\;\;\; Y \geq 0
\end{eqnarray}
where $B$ is Brownian motion, $\sigma$ and $b$ are borel functions, $\forall a \in \mathbb{R}$, $L_{t}^{a}(Y)$ is the local time of $Y$ at $a$. In the same subsection, we give also some explicit sufficient conditions which ensure the pathwise uniqueness for the SDE (\ref{Local}).\\
\underline{Subsection }(\ref{SDE with local time}):We establish both the pathwise uniqueness and uniqueness in law for the SDE (\ref{eqna}).\\
\underline{Subsection} (\ref{On One SDE with non-sticky boundary conditions}): Is dedicated to the pathwise uniqueness of SDE with a non sticky boundary condition. This result extends an early work of S. Manabe and T. Shiga \cite{Shiga}.\\
\underline{Section} (\ref{Tanaka}): Prokaj in \cite{prokaj} has showed a recent result on the pathwise uniqueness of the so-called perturbed Tanaka equation:
\begin{eqnarray}\label{tanaka}
Y_t= y+\int_0^t sign(Y_s)dM_t+ \lambda N_t
\end{eqnarray}
In section $5$, we use the local time technics introduced by Perkins \cite{perkins} and further developed by LeGall \cite{legall2}, to provide a simple proof of a more general result of this type.\\
Subsection \ref{Open}: We suggest some open problems .
\subsection{Preliminaries}
Throughout this paper we apply the usual conventions and currently standard notation of stochastic calculus associated with Itô integral.
In particular, we denote by $(\Omega, \mathcal{F}, (\mathcal{F}_{t})_{t \geq 0}; \mathbb{P})$ a filtered probability space, satisfying the usual conditions.\\
The notion of local time was introduced by P. Lévy for measuring the time spent by a diffusion process in the vicinity of a point. Following Ouknine $(1991)$, we can distinguish three local times at $a$ associated with the continuous semimartingale started at $x$:
\begin{itemize}
  \item the right local time at $a$ denoted $L_{t}^{a+}(X)$:
  $$\frac{1}{2}L_{t}^{a+}(X)=(X_{t}-a)^{+}-(x-a)^{+}-\int_{0}^{t}1_{\{X_{s}>a\}}\;dX_{s}$$
  \item the left local time at $a$ denoted $L_{t}^{a-}(X)$:
  $$\frac{1}{2}L_{t}^{a-}(X)=(X_{t}-a)^{-}-(x-a)^{-}-\int_{0}^{t}1_{\{X_{s}<a\}}\;dX_{s}$$
  \item the symmetric local time at $a$ denoted $L_{t}^{a}(X)$:
  $$L_{t}^{a}(X)=(L_{t}^{a+}(X)+L_{t}^{a-}(X))/2$$
\end{itemize}
\begin{prop}[Occupation times formula]
For any bounde measurable function $f$ and for all $t \geq 0$.
$$\int_{0}^{t}f(X_{s})\;d\langle X\rangle_{s}= \int_{\mathbb{R}}f(x)L_{t}^{x}(X)\;dx. $$
\end{prop}
\begin{prop}
For any $a$, the measure $dL^{a}(X)$ is a.s. carried by the set $\{t \geq 0: \; X_{t}=a\}$.
\end{prop}
\begin{prop}
There exists a modification of the random field \\$(L_{t}^{a}(X)\;; a \in \mathbb{R}, \; t \geq 0)$ such that the map $(a,t)\longrightarrow L_{t}^{a}(X)$ is a.s. continuous in $t$ and càdlàg in $a$. Moreover,
$$L_{t}^{a}(X)-L_{t}^{a-}(X)= 2 \int_{0}^{t} 1_{\{X_{s}=0\}}\;dX_{s}, \;\;\; t \geq0.$$
\end{prop}
We refer the reader to the book of Protter \cite{Protter} for a complete account on local times.
Through this paper  the indicator function is denoted $1$. We define $\vee$ and $\wedge$ through $a\vee b=max(a,b)$  and $a\wedge b=min(a,b)$. The positive and negative part is given respectively by $x^{+}=x\vee0$, $x^{-}=x\wedge 0$.
\section{Continuous case}\label{continuous cas}
In this section we derive a generalization of comparison theorem for local times of semimartingales proved in \cite{Ouk Nakao} in the case of continuous semimartingales. The proof used here does not need the upcrossing approximation of the local time process as in \cite{Ouk Nakao}. In order to give such an extension, we need first to recall  the balayage formula.
\\

At the beginning, let
 $X=(X_{t},\;\;\;t \geq 0)$ be a continuous $\mathcal{F}_{t}$-semimartingale issued from zero. For every $t>0$ we define
$$\gamma_{t}=sup\{s \leq t\;:\; X_{s}=0\},$$
with the convention $\sup(\emptyset)=0$, hence in particular $\gamma_{0}=0$. The random variables $\gamma_{t}$ are clearly not stopping times since they depend on the future.
\\

Let us recall an important result of Azéma-Yor\cite{Yor}:
\begin{prop}(Balayage Formula)
\begin{enumerate}
\item[(i)] Let $X$ be a continuous semimartingale, if $k$ is a locally bounded predictable process,
then
 $$k_{\gamma_{t}}X_{t}=k_{0}X_{0}+\int_{0}^{t}k_{\gamma_{s}}\;dX_{s},$$
and therefore $(k_{\gamma_{t}}X_{t})_{t \geq 0}$ is a continuous semimartingale.
Moreover, if $k$ is nonnegative then,
$$L_{t}^{0}(k_{\gamma_{.}}X_{.})=\int_{0}^{t}k_{s}\;dL_{s}^{0}(X).$$
\item [(ii)] If $X$ is a local martingale. $k_{\gamma}X$ is also a local martingale and its local time at $0$ is equal to
    $$\int_{0}^{t}k_{s}\;dL_{s}^{0}(X).$$
\end{enumerate}
\end{prop}
In the two following theorems we use balayage formula to give a clean proof of theorems originated from Nasyrov \cite{Nasyrov}.
\begin{thm}(The generalized Tanaka formula)
For $t \geq 0$, $z>0$, we have
$$\frac{1}{2}(z\wedge L_{t}^{0}(X))=1_{\{L_{t}^{0}\leq z\}}X_{t}^{+}-X_{0}^{+}-\int_{0}^{t}1_{\{X_{s}\geq 0\;,\;L_{s}^{0} \leq z\}}\;dX_{s}.$$
\end{thm}
\begin{proof}
Let us set $k_{t}=1_{\{L_{t}^{0}\leq z\}}$,\\
Balayage formula yields,
\begin{eqnarray*}
k_{\gamma_{t}}.X_{t}^{+}&=&k_{0}X_{0}^{+}+\int_{0}^{t}k_{\gamma_{s}}\;dX_{s}^{+}
\end{eqnarray*}
Using Tanaka's formula, this equality can be reexpressed as:
\begin{eqnarray*}
k_{\gamma_{t}}.X_{t}^{+}&=& X_{0}^{+}+\int_{0}^{t}1_{\{L_{\gamma_{s}}^{0}\leq z\}}1_{\{X_{s} \geq 0\}}\;dX_{s}+\frac{1}{2}\int_{0}^{t}1_{\{L_{\gamma_{s}}^{0}\leq z\}}\;dL_{s}^{0}(X).
\end{eqnarray*}
Since the measure $dL_{.}^{0}(X)$ is carried by the set $\{s\;,\;X_{s}=0\}$ and by the definition of $\gamma_{t}$, we see that $L_{\gamma_{t}}^{0}(X)=L_{t}^{0}(X)$, and then
$$1_{\{L_{t}^{0}\leq z\}}X_{t}^{+}=X_{0}^{+}+\int_{0}^{t}1_{\{X_{s}>0\;,\;L_{s}^{0}\leq z\}}\;dX_{s}+\frac{1}{2}(z \wedge L_{t}^{0}(X)).$$
Thus the result.

\end{proof}
\begin{rem}
It is obvious that $z=+\infty$ in precedent theorem, corresponds to the classical Tanaka formula.
\end{rem}
\begin{thm}(The generalized Skorokhod equation)
Let $\Phi$ be a locally integrable function. For $t\geq 0$,
\begin{eqnarray*}
\Phi(L_{t}^{0}(X))|X_{t}|=\Phi(0) |X_{0}|+\int_{0}^{t} sgn(X_{s}) \Phi(L_{s}^{0}(X))\;dX_{s}+\int_{0}^{L_{t}^{0}(X)} \Phi(z)\;dz.
\end{eqnarray*}
Moreover, if $\Phi$ is continuous and strictly positive function, we have
$$\int_{0}^{L_{t}^{0}(X)} \Phi(z)\;dz=-\min_{0\leq s \leq t} \;\min\left( \int_{0}^{s}sgn(X_{u})\Phi(L_{u}^{0}(X))\;dX_{u},0\right).$$
\end{thm}
\begin{proof}
Let us set $k_{t}=\Phi(L_{t}^{0}(X))$ we denote $\gamma_{t}$ the random variable $$\gamma_{t}=\sup\{s <t\;,\; X_{s}=0\}.$$
 \\Balayage formula gives,
$$k_{\gamma_{t}}|X_{t}|=k_{0}|X_{0}|+\int_{0}^{t} k_{\gamma_{s}}\;d|X_{s}|.$$
If we apply Tanaka's formula on $|X|$, we get:
\begin{eqnarray*}
k_{\gamma_{t}}|X_{t}|&=&k_{0}|X_{0}|+\int_{0}^{t} k_{\gamma_{s}}sgn(X_{s})\;dX_{s}+ \int_{0}^{t} k_{\gamma_{s}}\;d_{s}L_{s}^{0}(X)\\
&=&k_{0}|X_{0}|+ \int_{0}^{t} \Phi(L_{s}^{0}(X))sgn(X_{s})\;dX_{s}+ \int_{0}^{t} \Phi(L_{s}^{0}(X))\;d_{s}L_{s}^{0}(X)
\end{eqnarray*}
Since $(L_{t}^{0}(X),\;t \geq 0)$ is continuous nondecreasing process, the equality
$$\int_{0}^{t} \Phi(L_{s}^{0}(X))\;d_{s}L_{s}^{0}(X)=\int_{0}^{L_{t}^{0}(X)}\Phi(z)\;dz,$$
holds for any $t>0$. For more details see Nasyrov \cite{Nasyrov}.
Thus,
$$k_{\gamma_{t}}|X_{t}|=k_{0}|X_{0}|+ \int_{0}^{t} \Phi(L_{s}^{0}(X))sgn(X_{s})\;dX_{s}+\int_{0}^{L_{t}^{0}(X)}\Phi(z)\;dz$$,
whence the desired result.
\end{proof}
\\
Now, the inequality between local times combined with the balayage formula allow us to establish our main result:
\begin{thm}\label{positif}
Let $X$ and $Y$ be two continuous semimartingales such that
\begin{enumerate}
  \item[(1)] $\{s\ge 0\;/\;X_{s}=0\}\subset\{s\geq 0\;/\;Y_{s}=0\}$ and,
  \item [(2)]$X_{s}^{+}\leq Y_{s}^{+}$,
  \end{enumerate}
  then the measure $dL_{t}^{0}(X)$ is absolutely continuous with respect to the measure $dL_{t}^{0}(Y)$ :
  $$dL_{t}^{0}(X)\ll dL_{t}^{0}(Y),$$ and there exist a predictable process $(\theta_{s})_{s\geq 0}$, $\theta_{s} \in [0,1]$ such that
  $$L_{t}^{0}(X)=\int_{0}^{t}\theta_{s}\;dL_{s}^{0}(Y),$$
\end{thm}
\begin{proof}
We recall first that $L_{t}^{0}(Z)=L_{t}^{0}(Z^{+})$ for every semimartingale $Z$. Hence it is enough to consider the case where $X$ and $Y$ are non-negatives semimartingales.\\
 Now, let us remark that if
 $$\{s\geq 0\;/\;X_{s}=0\}\subset\{s\geq 0\;/\;Y_{s}=0\},$$ then, $$\gamma_{t}^{X}=\gamma_{t}^{Y}=\gamma.$$
Let $k$ be a predictable, positive and locally bounded process. Under the assumptions stated above, the processes
$$(\widetilde{X}_{t})_{t\geq 0}=(k_{\gamma_{t}}.X_{t})_{t\geq 0}\;\;\;\;\mbox{and}\;\;\;\;(\widetilde{Y}_{t})_{t\geq 0}=(k_{\gamma_{t}}.Y_{t})_{t\geq 0} $$
satisfy
$$\{s\geq 0\;/\;\widetilde{X}_{s}=0\}\subset\{s\geq 0\;/\;\widetilde{Y}_{s}=0\}.$$
It is clear that:
$$\widetilde{X}_{t}\leq \widetilde{Y}_{t}.$$
Following Ouknine \cite{Ouk Nakao}, under this hypothesis,
$$L_{t}^{0}(\widetilde{X})\ll L_{t}^{0}(\widetilde{Y}).$$
On the other hand, by the balayage formula:
$$L_{t}^{0}(\widetilde{X})=\int_{0}^{t}k_{s}\;dL_{s}^{0}(X)\;\;\;\;\; and \;\;\;\;\;L_{t}^{0}(\widetilde{Y})=\int_{0}^{t}k_{s}\;dL_{s}^{0}(Y) .$$
Consequently,
$$\int_{0}^{.}k_{s}\;dL_{s}^{0}(X) \ll \int_{0}^{.}k_{s}\;dL_{s}^{0}(Y). $$
Hence,
$$dL_{t}^{0}(X)\ll dL_{t}^{0}(Y).$$
Therefore, by Radon-Nikodym theorem, there exists a predictable process $\theta$, $\theta \in [0,1]$  such that
$$L_{t}^{0}(X)=\int_{0}^{t}\theta_{s}\;dL_{s}^{0}(Y).$$
\end{proof}
\\

\begin{rem}
The Radon-Nikodym derivative of $dL_{t}^{0}(X)$ with respect to the measure $dL_{t}^{0}(Y)$ is precisely,
$$\theta_{\gamma_{t}}=\liminf_{\varepsilon\searrow 0}\frac{ X_{\gamma_{t}+\varepsilon}}{Y_{\gamma_{t}+\varepsilon}}.$$
\end{rem}
\begin{rem}
If there exists a continuous semimartingale $\xi$ such that\\
$X=\xi.Y$,
the formula can be derived by a result of \cite{Sup}
giving the local time of product of two continuous semimartingales. In fact,
\begin{eqnarray*}
L_{t}^{0}(X)&=&L_{t}^{0}(\xi.Y)\\
&=& \int_{0}^{t} \xi_{s}^{+}dL_{s}^{0}(Y)+\int_{0}^{t}Y_{s}^{+}dL_{s}^{0}(\xi).
\end{eqnarray*}
The measure $dL_{s}^{0}(\xi)$ is carried by the set $\{s\;,\;\xi_{s}=0\}\subset\{s\;,\;X_{s}=0\} $.
According to the assumption $(1)$,
$$\{s\;,\;\xi_{s}=0\}\subset\{s\;,\;Y_{s}=0\}. $$
Consequently, $$L_{t}^{0}(X)=\int_{0}^{t} \xi_{s}^{+}\;dL_{s}^{0}(Y).$$
Hence de desired result.
\end{rem}
Now, we show that The absolute continuity of the measure $dL_s^0(X)$ with respect to the measure $dL_s^0(Y)$ is localizable.
\begin{prop}
Let $A$ denote the set $\{\forall s \in [0,T] \;\;\; 0\leq X_s \leq Y_s\}$. If
 $$\{s\ge 0\;/\;X_{s}=0\}\subset\{s\geq 0\;/\;Y_{s}=0\},$$
 then,
 $$dL_s^0(X)\ll dL_s^0(Y) \;\;\;\; \mbox{on the set} \;A$$
\end{prop}
\begin{proof}
Let $X$ and $Y$ be two continuous semimartingales. Let $\epsilon >0$,\\ $0<t \leq T$.
 We denote $M_t^X([0,\epsilon])$ the number of upcrossings of $\epsilon$ by $X$ on the interval $[0,t]$, (recall that $X_s$ is said to have an upcrossing of $\epsilon$ at $s_0>0$ if for some $x>0$, $X_s \leq  \epsilon$ in $(s_0-x,s_0)$ and $X_s \geq \epsilon$ in $(s_0, s_0+x)$). On the set $A$ we have
$$M_t^X([0,\epsilon]) \leq M_t^Y([0,\epsilon]);$$
 by multiplying this inequality by $2 \epsilon$, we get by Lévy result which says that $L_t^0(X)=\lim_{\epsilon \rightarrow 0}2 \epsilon M_t^X([0,\epsilon]) $
 $$L_t^0(X) \leq L_t^0(Y) \;\;\;\; \mbox{on the set} \;A$$
 To prove that $dL_t^0(X) \ll dL_t^0(Y)$ on the set $A$, we proceed by the same way as theorem $2.3$ by using the balayage formula.
\end{proof}
\bigskip
\\
The purpose of this paragraph, it to show that under suitable but weak conditions, comparison result for local times can be proved. The idea of the proof is the same as in \cite{Ouk Nakao}.
\\
Let $X$ be a continuous semimartingale null in $0$. We set
$$Z(X):=\{s \leq T;\; X_s=0\}$$
So,
$$Z(X)^c:=\cup_{n \geq 1}]g_n,d_n[$$
 where $]g_n,d_n[$ is the $n^e$ interval of excursion around $0$.\\
 Set
 $$M_n^X(t)=\sup_{s \in ]g_n,d_n[\cap [0,t]}(X_s^+)$$
 \begin{prop}
Let $X$ and $Y$ be two continuous semimartingales such that:
 \begin{enumerate}
 \item $Z(X)=Z(Y)$,
 \item $M_n^X(t) \leq M_n^Y(t)$, $\; \forall n  \in \mathbb{N}$, $\forall t \leq T$.
 \end{enumerate}
 Then, $$dL_t^0(X) \ll dL_t^0(Y).$$
 \end{prop}
 \begin{proof}
 Let $X$ and $Y$ be two continuous semimartingales. let $\epsilon >0$,\\ $0<t \leq T$.
 We denote $M_t^X([0,\epsilon])$ the number of upcrossings of $\epsilon$ by $X$ on the interval $[0,t]$.
 If $M_n^X(t)  \leq M_n^Y(t) $ for all $ n \in \mathbb{N}$, then we have:
 $$M_t^X([0,\epsilon]) \leq M_t^Y([0,\epsilon]);$$
 by multiplying this inequality by $2 \epsilon$, we get by Lévy result \cite{levy}
 $$L_t^0(X) \leq L_t^0(Y) .$$
 To prove that $dL_t^0(X) \ll dL_t^0(Y)$, we proceed by the same way as theorem (\ref{positif}) by using the balayage formula.
 \end{proof}
 \begin{cor}
 Let $X=(X_1,..., X_n)$ be a $n$-dimensional semimartingale, $N_1$ and $N_2$ be norms on $\mathbb{R}^n$ such that $N_1 \leq N_2$. Then, $dL_t^0(N_1(X)) \leq  dL_t^0(N_2(X))$.
 \end{cor}
\section{Càdlàg case}\label{cad case}
In \cite{coquet}, F. Coquet and Y. Ouknine give a new comparison theorem for local times concerning essentially càdlàg semimartingales. In this section, we focus our attention to extend this result to the measures associated to this local times.\\
We first introduce some notations which will be used in this section.\\
Let $M$ be a random predictable set, closed on the right containing $\{0\}$.
  For $t \geq 0$, we define
$$\gamma_{t}=sup\{s \leq t,\; s\in M\}.$$
Before we give statements of our main results, we recall this definition:
\begin{defn}
Let $Z$ be a semimartingale, we define the local time of $Z$, at $0$ as the continuous adapted increasing process $L^{0}(Z)$ satisfying
$$|Z_{t}|=|Z_{0}|+\int_{0}^{t}sgn(Z_{s^{-}})\;dZ_{s}+2 \sum_{0<s\leq t}\;[Z_{s}^{-}1_{\{Z_{s^{-}}>0\}}+Z_{s}^{+}1_{\{Z_{s^{-}}\leq 0\}}]+L_{t}^{0}(Z),$$
where we set $sgn(0)=-1$.
\end{defn}

The  balayage formula for càdlàg semimartingale is essential for the proof of our main result.
\begin{prop}
Let $(X_{t})_{t \geq 0}$ be a right continuous process such that $X$ is null on the $M$, and $k$ be bounded and predictable process, $(k_{\gamma_{t}}X_{t})_{t \geq 0}$ is a right continuous semimartingale and
$$k_{\gamma_{t}}X_{t}=k_{0}X_{0}+ \int_{0}^{t}k_{\gamma_{s}}\;dX_{s}.$$
Moreover,
$$L_{t}^{0}(k_{\gamma_{.}}X_{.})=\int_{0}^{t}|k_{s}|\;dL_{s}^{0}(X).$$
\end{prop}

Owing to Tanaka formula for càdlàg semimartingales, it is natural to establish our main result in discontinuous case under new hypothesis.
\begin{thm}
Let $X$ and $Y$  be two càdlàg semimartingales such that
\begin{enumerate}
  \item [(1)]$\{X=X_{-}=0\}\subset\{Y=Y_{-}=0\},$
  \item [(2)]$X^{+}\leq Y^{+}$.
  \end{enumerate}
  Then,
  $$dL_{t}^{0}(X)\ll dL_{t}^{0}(Y) \;\;\;\;  \mbox{for all}\; t>0.$$
  And there exists a predictable process $(\theta_{s})_{s\geq 0}$, $\theta_{s} \in [0,1]$ such that:
  $$L_{t}^{0}(X)=\int_{0}^{t}\theta_{s}\;dL_{s}^{0}(Y)$$
\end{thm}

\begin{proof}[Proof of theorem]
 Without loss of generality, it suffices to consider the case where $X$ and $Y$ are non-negative semimartingales. Let $CP(X)$ denotes the continuous part of a process $X$. It is well known that
\begin{eqnarray*}
\frac{1}{2}L_{t}^{0}(X)&=&CP\left[\int_{0}^{t}1_{\{X_{s^{-}}=0\}}\;dX_{s}\right]\\
&=&CP\left[\int_{0}^{t}1_{\{X_{s^{-}}=X_{s}=0\}}\;dX_{s}\right]
\end{eqnarray*}
the process $V_{t}=\int_{0}^{t}1_{\{X_{s^{-}}=0\}}\;dX_{s}$ is of bounded variation, so this integral
$$\int_{0}^{t}1_{\{X_{s^{-}}=X_{s}=0\}}\;dX_{s}$$
 is well defined.
 And we have:
\begin{eqnarray*}
\frac{1}{2}L_{t}^{0}(X)
&=& CP\left[\int_{0}^{t}1_{\{X_{s}=X_{s^{-}}=0\}}\;dY_{s}-\int_{0}^{t}1_{\{X_{s}=X_{s^{-}}=0\}}\;d(Y_{s}-X_{s})\right],\\
\end{eqnarray*}
see \cite{Sup} for more details.
Using this assumption:
$$\{X=X_{-}=0\}\subset\{Y=Y_{-}=0\},$$
we obtain this inequality:
\begin{eqnarray*}
\frac{1}{2}L_{t}^{0}(X)&\leq &CP\left[\int_{0}^{t}1_{\{Y_{s}=Y_{s^{-}}=0\}}\;dY_{s}\right]-CP\left[\int_{0}^{t}1_{\{X_{s}=X_{s^{-}}=0\}}\;d(Y-X)_{s}\right],
\end{eqnarray*}
and the same assumption allows us to write:
\begin{eqnarray*}
\int_{0}^{t}1_{\{X_{s}=X_{s^{-}}=0\}}\;d(Y-X)_{s}&=& \int_{0}^{t}1_{\{X_{s}=X_{s^{-}}=0\}}1_{\{Y_{s^{-}}-X_{s^{-}}=0\}}\;d(Y-X)_{s}\\
&=&\frac{1}{2}\int_{0}^{t}1_{\{X_{s^{-}}=X_{s}=0\}}\;dL_{s}^{0}(Y-X)+\int_{0}^{t}1_{\{X_{s^{-}}=X_{s}=0\}}\;dA_{s}
\end{eqnarray*}
where
$$A_{t}= \sum_{s\leq t}1_{\{Y_{s^{-}}\leq X_{s^{-}}\}}(Y_{s}-X_{s})^{+}+1_{\{Y_{s^{-}}> X_{s^{-}}\}}(Y_{s}-X_{s})^{-}.$$
Therefore, $\int_{0}^{t}1_{\{X_{s^{-}}=X_{s}=0\}}\;d(Y-X)_{s}\geq 0$, this clearly forces:
$$\frac{1}{2}L_{t}^{0}(X)\leq  CP\left(\int_{0}^{t}1_{\{Y_{s^{-}}=0\}}\;dY_{s}\right),$$
which establishes the following formula:
$$L_{t}^{0}(X)\leq L_{t}^{0}(Y).$$
To complete the proof, we set
$$M=\{X_{s^{-}}=X_{s}=0\}.$$
Using the balayage formula we can define the processes:
$$\widetilde{X}_{t}=k_{\gamma_{t}}X_{t}=\int_{0}^{t}k_{\gamma_{s}}\;dX_{s}\;\;\;\;\mbox{and}\;\;\;\;\widetilde{Y}_{t}=k_{\gamma_{t}}Y_{t}=\int_{0}^{t}k_{\gamma_{s}}\;dY_{s}$$
which are càdlàg semimartingales satisfying the assumptions
$(1)$ and $(2)$.\\
Consequently,
$$L_{t}^{0}(\widetilde{X})\leq L_{t}^{0}(\widetilde{Y}).$$
Therefore, we obtain
$$\int_{0}^{t}k_{s}\;dL_{s}^{0}(X) \leq \int_{0}^{t}k_{s}\;dL_{s}^{0}(Y),$$
Thus, the result follows by a particular choice of the process $k$.
\end{proof}
\section{Application to SDE }\label{Application to SDE}
\subsection{Pathwise uniqueness for some reflected SDE's}\label{SDE local}
We deal this section by recalling briefly some type of uniqueness of solution to stochastic differential equation. For more details about this  notion, we refer the reader to \cite{Revuz}.\\
We consider the following stochastic differential equation:
$$X_{t}=X_{0}+ \int_{0}^{t} \sigma(X_{s})\;dB_{s}+\int_{0}^{t} b(X_{s})\;ds, \;\;\;\;\;e(\sigma,b)$$
where $X_{0}$ is a random variable and $b$, $\sigma$ are two Borel functions.\\ We suppose that, $\forall t \geq 0$,
$$\int_{0}^{t} \sigma^{2}(X_{s})\;ds< +\infty\;,\;\;\;\;\;\; \int_{0}^{t} |b(X_{s})|\;ds< +\infty.$$
We associate to $e(\sigma,b)$ another equation $e^{'}(\sigma,b)$ by addition of local time component:
\begin{equation*}
\label{eql}
\left\{
\begin{array}{lll}
dY_{t}=\sigma(Y_{t})\;dB_{t}+b(Y_{t})\;dt+\frac{1}{2}dL_{t}^{0}(Y)\;\;\;\;\;\;\;e^{'}(\sigma,b)\\
Y_{t} \geq 0 \\
\end{array}
\right.
\end{equation*}
\begin{defn}
\begin{enumerate}
\item A solution of the SDE $e(\sigma,b)$ (or $e^{'}(\sigma,b)$) is a pair $(X,B)$ of adapted processes defined on a probability space $(\Omega, (\mathcal{F}_{t})_{t \geq 0},\mathbb{P})$ such that
\begin{itemize}
  \item $B$ is a standard  $\mathcal{F}$-Brownian motion in $\mathbb{R}$,
  \item X satisfies $e(\sigma,b)$ or $e^{'}(\sigma,b)$ and the conditions above.
\end{itemize}
\item A  solution $X$ of the SDE $e(\sigma,b)$ or $e^{'}(\sigma,b)$ is said to be trivial if
$$\mathbb{P}(\{X_{t}=X_{0},\;\; \forall t \geq 0\})=1.$$
\end{enumerate}
\end{defn}
In Proposition $3.2$, Chap. IX of Ref. \cite{Revuz}, it is shown that if uniqueness in law holds for $e(\sigma,b)$ and if the local time $L^{0}(X^{1}-X^{2})=0$ for any pair of solutions such that $X_{0}^{1}=X_{0}^{2}$ a.s., then the pathwise uniqueness holds for $e(\sigma,b)$.\\
The analogue of this result for the SDE $e^{'}(\sigma,b)$ is given by this theorem:
\begin{thm}\label{prin}
The two following properties are equivalent:
\begin{enumerate}
  \item There is pathwise uniqueness for equation $e^{'}(\sigma,b)$.
  \item There is uniqueness in law for equation $e^{'}(\sigma,b)$ and whenever $(X,B)$ and $(Y,B)$ are two solutions such that $X_{0}=Y_{0}$ a.s.,$\;dL_{s}^{0}(X-Y)$ is carried by the set $\{s,\; X_{s}=Y_{s}=0\}$.
\end{enumerate}
\end{thm}
\begin{proof}
\\
$(1)\Rightarrow(2)$. Trivial.\\
$(2)\Rightarrow(1)$. We will prove that whenever $X$ and $Y$ are solutions of $e^{'}(\sigma,b)$, $X\vee Y$ and $X\wedge Y$ are also solutions.\\
By Tanaka's formula one has:
\begin{eqnarray*}
(X_{t}\vee Y_{t})&=&Y_{t}+ (X_{t}-Y_{t})^{+}\\
&=& Y_{t}+\left(\int_{0}^{t}1_{\{X_{s}>Y_{s}\}}\;d(X_{s}-Y_{s})+\frac{1}{2}L_{t}^{0}(X-Y)\right)\\
&=&\int_{0}^{t} \sigma(X_{s}\vee Y_{s})\;dB_{s}+ \int_{0}^{t}b(X_{s}\vee Y_{s})\;ds\\
&+& \frac{1}{2}\int_{0}^{t}1_{\{X_{s}>Y_{s}\}}\;dL_{s}^{0}(X)+\frac{1}{2}\int_{0}^{t}1_{\{X_{s}\leq Y_{s}\}}\;dL_{s}^{0}(Y)+ \frac{1}{2}L_{t}^{0}(X-Y).
\end{eqnarray*}
Observe that by a support argument for local time we have: $\forall t \geq 0$,
$$\int_{0}^{t}1_{\{X_{s}>Y_{s}\}}\;dL_{s}^{0}(X)=\int_{0}^{t}1_{\{0>Y_{s}\}}\;dL_{s}^{0}(X)\;\;\;\mbox{and} \;\;\;\int_{0}^{t}1_{\{X_{s}\leq Y_{s}\}}\;dL_{s}^{0}(Y)=\int_{0}^{t}1_{\{X_{s}\leq 0\}}\;dL_{s}^{0}(Y).$$
Since the measure $dL_{s}^{0}(X-Y)$ is carried by the set $\{s, \;X_{s}=Y_{s}=0\}$, $L_{t}^{0}(X-Y)$ can be expressed as:
$$L_{t}^{0}(X-Y)=\int_{0}^{t}1_{\{X_{s}=Y_{s}=0\}}\;dL_{s}^{0}(X-Y).$$
Now from a result by Ouknine  and Rutkowski \cite{Rutkowski}, in which they prove the local time of sup of two semimartingales:
$$L_{t}^{0}(X\vee Y)=\int_{0}^{t}1_{\{X_{s} < 0\}} dL_{s}^{0}(Y)+\int_{0}^{t}1_{\{Y_{s}\leq  0\}}dL_{s}^{0}(X)+ \int_{0}^{t}1_{\{X_{s}=Y_{s}= 0\}}dL_{s}^{0}(Y-X). $$
Hence
$X\vee Y$ is a solution to $e{'}(\sigma,b)$. Similarly we can show that $X\wedge Y$ is a solution as well.
As $X$ and $Y$ have integrable paths on finite time interval, then
\begin{eqnarray*}
\mathbb{E}[|X_{t}- Y_{t}|]&=&\mathbb{E}[X_{t}\vee Y_{t}]-\mathbb{E}[X_{t}\wedge Y_{t}]
\end{eqnarray*}
and by uniqueness in law, we obtain:
$$\mathbb{E}[|X_{t}- Y_{t}|]=0,$$
then, $X$ and $Y$ are indistinguishable.
\end{proof}
\\

In the next step we show that under some hypothesis on the diffusion coefficient $\sigma$, we state some sufficient conditions for the pathwise uniqueness of the equation $e^{'}(\sigma,b)$. Let us first recall an important result of Ouknine \cite{Ident}:
\begin{prop}
Suppose $\sigma$ and $b$ are two bounded borel functions.
If $\sigma$ verifies the following condition:
\begin{enumerate}
  \item  $|\sigma|\geq \varepsilon >0$
  \item there exists $n \in 2 \mathbb{N}$ such that
$$|x^{n}\sigma(x)-y^{n}\sigma(y)|^{2}\leq \rho (|h(x)-h(y)|),\;\;\mbox{for all}\; x, y \in \mathbb{R}$$
where $\rho:\left[ 0,\infty \right)\rightarrow \left[ 0,\infty \right)$ is continuous and nondecreasing, $\rho(0)=0,\\\rho(x)>0 $ for $x>0$, and
$\rho(\alpha x)\leq  \alpha \rho(x)\;\;$  $\forall \alpha >1\;$,  $\forall x>0$, and
\[
  \int^{\epsilon}_{0}\frac{du}{\rho(u)}=+\infty \mbox{ for some}\;\epsilon >0,
\]
\end{enumerate}
then we have pathwise uniqueness for equation $e^{'}(\sigma,b)$.

\end{prop}
\begin{rem}
In $1988$, to prove the pathwise uniqueness of solutions of  $e^{'}(\sigma,b)$, Ouknine \cite{Ouk Nakao} tried to prove that $L_{t}(X-Y)=0$ whenever $X$ and $Y$ are solutions. In this time, he has just proved that
$$\int_{0}^{t}(X_{s}^{2p}+Y_{s}^{2p})\;dL_{s}^{0}(X-Y)=0,$$
which with the Theorem \ref{prin} is sufficiently enough to prove the pathwise uniqueness of $e^{'}(\sigma,b)$.
\end{rem}
We can prove the same result as above but under weaker conditions.
\begin{prop}\label{propn}
If $\sigma$ verify the following conditions
\begin{enumerate}
  \item  $|\sigma|\geq \varepsilon >0$,
  \item there exists $n \in \mathbb{N}$, $c>0$ such that
$$|x^{n-1}\sigma(x)-y^{n-1}\sigma(y)|^{2}\leq c\rho (|x^{n}-y^{n}|)\;\;\;\; \forall x\;,y \in \mathbb{R},$$
where  $\rho:\left[ 0,\infty \right)\rightarrow \left[ 0,\infty \right)$ is continuous and nondecreasing, $\rho(0)=0,\\\rho(x)>0 $ for $x>0$, and
$$
\int^{\epsilon}_{{0}^{+}}\frac{du}{\rho(u)}=+\infty \;\;\;\mbox{ for some }\;\;\epsilon >0,$$
\end{enumerate}
then we have pathwise uniqueness for equation $e^{'}(\sigma,b)$. \\
\end{prop}
\begin{proof}
Let $X$ and $Y$ be two solutions of $e^{'}(\sigma,b)$  defined on a common probability basis with respect to same Brownian motion such that $X_{0}=Y_{0}$.
 The hypothesis $(1)$ is sufficient for  uniqueness in law.\\ The task now is to prove that the measure $\;dL_{s}(X-Y)$ is carried by the set $\{s,\; X_{s}=Y_{s}=0\}$. Let us show first that $L_{.}^{0}(X^n-Y^n)=0$. \\
Applying the assumption $(2)$ combined with Itô's formula,  we obtain:
\begin{eqnarray*}
\int_{0}^{t}1_{\{X_{s}^{n}-Y_{s}^{n}>0\}}\frac{d\langle X^{n}-Y^{n}\rangle_{s}}{\rho(X_{s}^{n}-Y_{s}^{n})}&\leq& n^{2}\int_{0}^{t}\frac{ (X_{s}^{n-1}\sigma(X_{s})-Y_{s}^{n-1}\sigma(Y_{s}))^{2}}{\rho(X_{s}^{n}-Y_{s}^{n})}1_{\{X_{s}^{n}-Y_{s}^{n}>0\}} \;ds\\
&\leq & n^{2}ct.
\end{eqnarray*}
Hence, by the occupation times formula, we get:
$$\int_{0}^{t}1_{\{X_{s}^{n}-Y_{s}^{n}>0\}}\frac{d\langle X^{n}-Y^{n}\rangle_{s}}{\rho(X_{s}^{n}-Y_{s}^{n})}=\int_{0^{+}} \frac{da}{\rho(a)}L_{t}^{a}(X^{n}-Y^{n})< \infty.$$
Thus with the help of the right continuity of $a\rightarrow L_{t}^{a}(X^{n}-Y^{n})$ and the condition $\int^{\epsilon}_{{0}^{+}}\frac{du}{\rho(u)}=\infty \;\mbox{for some }\;\epsilon >0$, $L_{t}^{0}(X^{n}-Y^{n})=0.$\\
On the other hand, $X^{n}-Y^{n}$ can be written as :
$$X^{n}-Y^{n}=(X-Y)P(X,Y),$$
with $P(X,Y)=\sum_{k=0}^{n-1}X^{k}Y^{n-k-1}$. So,
$$L_{t}^{0}(X^{n}-Y^{n})=L_{t}^{0}[(X-Y)P(X,Y)]\;\;\; \forall n \in \mathbb{N}^*$$
Using the formula giving the local time of product of two semimartingales which goes back to Ouknine\cite{Sup}, we obtain:
\begin{eqnarray}\label{car}
L_{t}^{0}(X^{n}-Y^{n})&=&L_{t}^{0}[(X-Y)P(X,Y)]\\ \nonumber
&=& \int_{0}^{t}(X_{s}-Y_{s})^+\;dL_{s}^{0}(P(X,Y))+\int_{0}^{t}P(X_{s},Y_{s}) \;dL_{s}^{0}(X-Y). 
\end{eqnarray}
But $\int_{0}^{t}(X_{s}-Y_{s})\;dL_{s}^{0}(P(X,Y))=0$, since the measure $dL_{s}^{0}(P(X,Y))$ is carried by the set $\{s,\; P(X,Y)=0\}=\{s,\; X_{s}=Y_{s}=0\}$,
the last equality is due to the positivity of $X$ and $Y$.\\
(\ref{car}) now becomes :
\begin{eqnarray*}
\int_{0}^{t}P(X_{s},Y_{s}) \;dL_{s}^{0}(X-Y)&=&n \int_{0}^{t}X^{n-1}\;dL_{s}^{0}(X-Y)\\
&=&0.
\end{eqnarray*}
Finally, owing to the positivity of $X$, the measure $dL_{s}^{0}(X-Y)$  is carried by the set of $\{s,\;X_{s}=Y_{s}=0\}$.
\end{proof}
\begin{prop}
Let $X$ and $Y$ be two solutions of equation $e^{'}(\sigma,b)$. If $\sigma$ satisfies the assumptions of the previous proposition, $Z:= X-Y$ , and \\$Z^{'}:= X^{n}-Y^{n}$. Then,
$$dL_{t}(Z)\perp dL_{t}(Z^{'}).$$
\end{prop}
\begin{proof}
By the same argument in the proof of Proposition \ref{propn},
$$\int_{0}^{t}1_{\{X_{s}=Y_{s}=0\}}\;dL_{s}^{0}(X^{n}-Y^{n})=0.$$
But owing to assumptions satisfied by $\sigma$, the measure $dL_{s}^{0}(X-Y)$ is carried by the set $\{s,\;X_{s}=Y_{s}=0\}$ which completes the proof.
\end{proof}
\\

By the same lines as the previous proof of proposition \ref{propn} we have the following more general setting:
\begin{prop}\label{propf}
There is pathwise uniqueness for equation $e^{'}(\sigma,b)$ if the two following conditions hold:
\begin{enumerate}
  \item[(i)] There is uniqueness in law for equation $e^{'}(\sigma,b)$,
  \item [(ii)]
$\sigma$ verifies the following condition:
$$|f^{'}(x)\sigma(x)-f^{'}(y)\sigma(y)|^{2}\leq  c\rho( |f(x)-f(y)|)\;\;\;\mbox{for all}\;\; x,\;y \in \mathbb{R},$$
where $f:\mathbb{R}\rightarrow \mathbb{R}$ is strictly increasing, continuously differentiable, and it is a difference  of two convex functions, its derivative $f^{'}$ satisfies:
 $$f^{'}(x)=0\Leftrightarrow x=0,$$ and
 $\rho:\left[ 0,\infty \right)\rightarrow \left[ 0,\infty \right)$ is continuous and nondecreasing, $\rho(0)=0,\\\rho(x)>0 $ for $x>0$, such that:
  \[
  \int^{\epsilon}_{0}\frac{du}{\rho(u)}=\infty \;\mbox{ for some }\;\epsilon >0.
\]
\end{enumerate}
\end{prop}
The following corollary is essential for the proof of Proposition \ref{propf}\;(see Ouknine and Rutkowski \cite{Rutkowski}).
\begin{cor}\label{corf}
Suppose that the function $g: \mathbb{R}\longrightarrow  \mathbb{R}$ is strictly increasing, continuously differentiable, and $g$ is a difference  of two convex functions, then we have:
$$L_{t}^{0}(g(X)-g(Y))= \int_{0}^{t}g^{'}(X_{s})\;dL_{s}^{0}(X-Y).$$
\end{cor}
\begin{proof}[Proof of Proposition \ref{propf}]
Let $(X,B)$ and $(Y,B)$ two solutions of the SDE $e^{'}(\sigma,b)$ such that $X_{0}=Y_{0}$.\\
Let assumptions $(i)$, $(ii)$ hold. Then:
\begin{eqnarray*}
\int_{0}^{t}1_{\{f(X_{s})-f(Y_{s})>0\}}\frac{d\langle f(X)-f(Y)\rangle_{s}}{\rho(f(X_{s})-f(Y_{s}))}&\leq & \int_{0}^{t}1_{\{f(X_{s})-f(Y_{s})>0\}}\frac{ (f^{'}(X_{s})\sigma(X_{s})-f^{'}(Y_{s})\sigma(Y_{s}))^{2}}{\rho(f(X_{s})-f(Y_{s}))}\;ds\\
&\leq & ct
\end{eqnarray*}
By the same arguments as in the last paragraph of the proof of the Proposition (\ref{propn}):
$$L_{t}^{0}(f(X)-f(Y))=0.$$
On account of Corollary(\ref{corf}), we have:
$$\int_{0}^{t}f^{'}(X_{s})\;dL_{s}^{0}(X-Y)=0.$$
 Since $f^{'}(x)= 0 \Leftrightarrow x=0$, we can conclude that the measure
 $dL_{t}^{0}(X-Y)$ is carried by the set $\{t/\;X_{t}=Y_{t}=0\}$.
\end{proof}
\begin{rem}
We insist that in case $(ii)$ $\sigma$ does not depend on $s$ as in Proposition (\ref{propn}).
\end{rem}
Now, we introduce the definition of $(LT)_{+}$ condition.
\begin{defn}
$\sigma$ satisfies $(LT)_{+}$ if whenever $V^{1}$ and $V^{2}$ are continuous adapted processes of bounded variation
on some $(\Omega, \mathcal{F},\mathbb{P})$ and $X^{i}$ $(i=1,2)$ are \textbf{positive} adapted solutions of
$$X_{t}^{i}=x_{i}+\int_{0}^{t} \sigma(X_{s}^{i})\;dB_{s}+ V_{s}^{i}\;\;\;\; i=1,2$$
$(x_{1},x_{2} \in \mathbb{R})$, then $L_{t}^{0}(X^{1}-X^{2})=0$ for all $t \geq 0$.
\end{defn}
Under $(LT)_{+}$ condition the pathwise uniqueness fails for solutions to the SDE $e(\sigma,b)$. However we can show that this uniqueness is valid for its absolute value.
\begin{prop}
Suppose that $b$ is an odd bounded measurable function on $\mathbb{R}$. Suppose that $\sigma$ is an odd bounded measurable function on $\mathbb{R}$ with $a^{-2}$ locally integrable on $\mathbb{R}$  and satisfies $(LT)_{+}$ condition.
 Then for any two weak solutions $X$ and $Y$ to this SDE:
\begin{eqnarray}\label{eqabs}
dX_{t}=\sigma(X_{t})\;dB_{t}+b(X_{t})\;dt
\end{eqnarray}
with a common Brownian motion  on a common probability space with\\ $X_{0}=Y_{0}$. We have the pathwise uniqueness  for this reflected stochastic differential equation:
$$d|X_{t}|=\sigma(|X_{t}|)\;dB_{t}+b(|X_{t}|)\;dt+\frac{1}{2}dL_{t}^{0}(|X|)$$
where $L_{t}^{0}(X)$ is the local time of $X$.
\end{prop}
\begin{proof}
Suppose that $\sigma$ and $b$ satisfy the assumptions of the proposition. Let $X$ and $Y$ two weak solutions to (\ref{eqabs}) with a common Brownian motion on a common probability space and $X_{0}=Y_{0}$. $X$ satisfies (\ref{eqabs}) and $\sigma$ is odd, then by Tanaka's formula
$$d|X_{t}|=\sigma(|X_{t}|)\;dB_{t}+b(|X_{t}|)\;dt+\frac{1}{2}dL_{t}^{0}(|X|))$$
 Similarly $|Y_{t}|$ satisfies equation with $Y_{t}$ in the place of $|X|$. By $(LT)_+$ condition we get the desired result.
\end{proof}
\vspace{0.1cm}
\\

It is shown by Barlow \cite{Barlow} that
$$dX_{t}=(a 1_{\{X_{t}>0\}}-b 1_{\{X_{t}\leq 0\}})\,dB_{t}$$
may not have any strong solutions if $a$ and $b$ are strictly positive. Then the pathwise uniqueness fails. However, we will  show the pathwise uniqueness for $\widetilde{X}=\frac{1}{a}X^{+}+\frac{1}{b}X^{-}$.
\begin{prop}
Let $X$ and $Y$ be two weak solutions to this SDE:
\begin{eqnarray}\label{eqbar}
dX_{t}= (a 1_{\{X_{t}>0\}}-b1_{\{X_{t}\leq 0\}})\;dB_{t}
\end{eqnarray}
with a common Brownian motion  on a common probability space such that $X_{0}=Y_{0}$.
then we have:
$$\mathbb{P}\left(\frac{1}{a}X_{t}^{+}+\frac{1}{b}X_{t}^{-}=\frac{1}{a}Y_{t}^{+}+\frac{1}{b}Y_{t}^{-}\; \mbox{for all} \;t\geq 0\right)=1$$
\end{prop}
\begin{proof}
Let $(X,B)$ and $(Y,B)$ be two weak solutions to (\ref{eqbar}) driving by the same Brownian motion $B$ on a common probability space, $X_{0}=Y_{0}$. Let $\alpha>0$, $\beta >0$. We set $\phi(X_{t})=\alpha X_{t}^{+}+ \beta X_{t}^{-} $\\
By Itô's formula we have
\begin{eqnarray}\label{phi}
\phi(X_{t})&=&\int_{0}^{t} (\alpha a 1_{\{X_{s}>0\}}+\beta b 1_{\{X_{s}\leq 0\}})\;dB_{s}+\frac{1}{2}(\alpha+\beta)L_{t}^{0}(X)
\end{eqnarray}
Tanaka's formula gives
\begin{eqnarray*}
L_{t}^{0}(\phi(X))&=& 2\int_{0}^{t} 1_{\{\phi(X_{s})=0\}}d\phi(X_{s})\\
&=& 2 \int_{0}^{t}\alpha 1_{\{X_{s}=0\}}\;dX_{s}^{+}+2 \int_{0}^{t}\beta1_{\{X_{s}=0\}}\;dX_{s}^{-}\\
&=& (\alpha+\beta)L_{t}^{0}(X).
\end{eqnarray*}
Setting $\alpha=\frac{1}{a}$, $\beta=\frac{1}{b}$, and substituting in (\ref{phi}), we get
\begin{eqnarray}\label{Barl}
\phi(X_{t})= B_{t}+\frac{1}{2}L_{t}^{0}(\phi(X)).
\end{eqnarray}
Thus, denoting by $\phi(Y)$ the process obtained in the same way from $Y$, and so $\phi(Y)$ satisfies equation (\ref{Barl}) with $\phi(Y)$ in place of $\phi(X)$. Since $\frac{1}{2}<1$, the equation (\ref{Barl}) has unique strong solution (see \cite{Harisson}). Consequently,
$$\mathbb{P}(\phi(X_{t})=\phi(Y_{t})\; \mbox{for all} \;t\geq 0)=1$$.
\end{proof}
\subsection{SDE with local time}\label{SDE with local time}
In this subsection, the main subject of this part, the stochastic differential equations involving local times of unknown process, is treated. A special case of such equations was introduced by Harisson and shepp \cite{Harisson}, they have established that the equation:
$$X_{t}=B_{t}+ \beta L_{t}^{0}(X),$$
has no solution if $|\beta|>1$, and has a unique strong solution when $|\beta|\leq 1$. Next Le Gall \cite{Legall} has studied equation of a form :
\begin{eqnarray}
X_{t}&= X_{0}+ \int_{0}^{t}\sigma_{s}(X_{s})\;dB_{s}+\int_{\mathbb{R}}L_{t}^{a}(X)\; \nu(da)
\end{eqnarray}
where $\nu$ stands for a finite  signed measure in $\mathbb{R}$, and the diffusion coefficient is assumed to be strictly positive function of finite variation.\\
We shall prove pathwise uniqueness of solutions to more general equation, corresponding to non Markovian type:
\begin{eqnarray}\label{eqnu}
X_{t}&=&X_{0}+ \int_{0}^{t}\sigma_{s}(X_{.},B_{.})\;dB_{s}+\int_{\mathbb{R}}L_{t}^{a}(X)\; \nu(da)
\end{eqnarray}
where
\begin{itemize}
\item[(i)]$L^{a}$ is the local time at $a$ for the semimartingale $X$ and $\nu$ is a $\sigma$-finite measure with $|\nu|(\{a\}) < 1$, $\forall a \in \mathbb{R}$,
\item[(ii)] $B$ is a standard Wiener process,
\item [(iii)]the mapping function $\sigma:\mathbb{R}^{+}\times C(\mathbb{R}^{+},\mathbb{R})\times C(\mathbb{R}^{+},\mathbb{R})\rightarrow \mathbb{R}$ is measurable and adapted to the filtration $(\mathcal{C}_{t})_{t \geq 0}$.
\end{itemize}
\begin{rem}
If $\sigma\equiv 1$ and $\nu$ is concentrated in the point $0$ with \\$|\nu|(\{0\}) < 1$, equation (\ref{eqnu}) describes the skew Brownian motion which was treated by several authors, e.g. Harisson and Shepp $(1981)$, Le Gall $(1982)$, Ouknine $(1991)$.
\end{rem}
\begin{rem}
In a certain sense, the jump condition:
$$|\nu|(\{a\}) < 1$$
is a natural restriction since there is, in general, no solution of equation  (\ref{eqnu}) if $\nu(\{a\}) <-1$ for some $a \in \mathbb{R}$ (cf. \cite{Harisson}).
\end{rem}
Since we are interested by stochastic differential equations with a general mapping functions $\sigma$. We shall generalize the definition of $(LT)$ condition introduced in Barlow and Perkins \cite{Barlow}.
\begin{defn}
$\sigma$ satisfies $(LT)$ if whenever $V^{1}$ and $V^{2}$ are continuous adapted processes of bounded variation
on some $(\Omega, \mathcal{F},\mathbb{P})$ and $X^{i}$ $(i=1,2)$ are adapted solutions of
$$X_{t}^{i}=x_{0}+\int_{0}^{t} \sigma_{s}(X^{i},B)\;dB_{s}+ V_{s}^{i}\;\;\;\; i=1,2$$
where $\sigma:\mathbb{R}^{+}\times C(\mathbb{R}^{+},\mathbb{R})\times C(\mathbb{R}^{+},\mathbb{R})\rightarrow \mathbb{R}$ is measurable and adapted to the filtration $(\mathcal{C}_{t})_{t \geq 0}$. then $L_{t}^{0}(X^{1}-X^{2})=0$ for all $t \geq 0$.
\end{defn}
The $(LT)$ condition will help us to get pathwise uniqueness for the SDE (\ref{eqnu}).
\begin{thm}\label{EDSb}
If $\sigma$ satisfies $(LT)$ condition, we have the following properties:
\begin{enumerate}
  \item[(i)] The solution to (\ref{eqnu}) is pathwise unique.
  \item [(ii)]There is uniqueness in law for equation (\ref{eqnu}).
\end{enumerate}
\end{thm}
\begin{proof}
\\
$(i)\;$ Let $(X,B )$ and $(X^{'},B)$  be two solutions of the SDE (\ref{eqnu})  with the same Brownian motion $B$ on the same probability space.\\
After localization, one can suppose that $X$ is bounded and $\nu$ is $\sigma$-finite measure because $a\rightarrow L_{t}^{a}(X)$ is compactly supported. The main tool to do this proof is the following transformation. Let us set
$$f_{\nu}(y)=e^{-2 \nu^{c}]0,y]}\prod_{z \leq y}\left(\frac{1-\nu\{z\}}{1+\nu\{z\}}\right),$$
where $\nu^{c}$ is the continuous part of the measure $\nu$.
It is well known that the function $f_{\nu}$ satisfies
$$0<m\leq f_{\nu}(x)\leq M \;\;\;\forall x \in \mathbb{R}$$
for some constants $m$, $M$. We set
$$F_{\nu}(x)=\int_{0}^{x}e^{-2 \nu^{c}]0,y]}\prod_{z \leq y}\left(\frac{1-\nu\{z\}}{1+\nu\{z\}}\right)\;dy,$$
hence the function $F_{\nu}$ is increasing, bijective and is a difference of two convex functions. We set
$$Y=F_{\nu}(X),\;\;\;\;\;\;\mbox{and}\;\;\;\;\mbox{and}\;\;\;\;Y^{'}=F_{\nu}(X^{'}).$$
By application of Itô-Tanaka's  formula one has:
$$Y_{t}=F_{\nu}(x)+\int_{0}^{t}\widetilde{\sigma}_{s}(X,B)\;dB_{s},$$
and,
$$Y^{'}_{t}=F_{\nu}(x)+\int_{0}^{t}\widetilde{\sigma}_{s}(X^{'},B)\;dB_{s}.$$
where $\widetilde{\sigma}_{s}(X,B)=f_{\nu}(X_{s})\sigma_{s}(X_{.}, B_{.})$.\\
Plainly,
$$m(X_{t}-X_{t}^{'})^{+} \leq (Y_{t}-Y^{'}_{t})^{+} \leq M (X_{t}-X_{t}^{'})^{+},$$
By comparison theorem for local times (see \cite{Sup}):
$$L_{t}^{0}(Y-Y^{'})\leq M L_{t}^{0}(X-X^{'}).$$

As $\sigma$ verifies $(LT)$ condition, we have
$$L_{t}^{0}(X-X^{'})=0.$$
It follows easily that $L_{t}^{0}(Y-Y^{'})=0$. Moreover, $Y-Y^{'}$ is continuous martingale and $Y_{0}-Y_{0}^{'}=0$, and so $Y=Y^{'}$. This  implies that the paths of $X$ are uniquely determined.
\\
$(ii)$ To prove the uniqueness in law, we use the theorem of Engelbert-Yamada-Watanbe \cite{ENgelbert-Yamada} which shows that the pathwise uniqueness implies uniqueness in law.
\end{proof}
\vspace{0.1cm}
\\

As an immediate application of the previous theorem, we have:
\begin{cor}
Let $b$ and  $\sigma$ be two measurable functions on $\mathbb{R}$. By $N_{\sigma}$ we denote the set $N_{\sigma}=\{x: \sigma(x)=0\}$ of zero of $\sigma$ .\\
If
\begin{enumerate}
\item  $\sigma$  satisfies $(LT)$ condition,
\item $\frac{b}{\sigma^{2}}1_{N_{\sigma}^{c}}\in L_{loc}^{1}(\mathbb{R})$.
\end{enumerate}
Then the pathwise uniqueness holds for:
\begin{equation}\label{eqb}
\left\{
\begin{array}{lll}
dX_{t}=\sigma(X_{t})\;dB_{t}+b(X_{t})\;dt\\
\int_{0}^{\infty}1_{N_{\sigma}}(X_{s})\;ds=0 \\
\end{array}
\right.
\end{equation}

\end{cor}
\begin{proof}
Let $X$ be a solution of (\ref{eqb}) for which the time spent at $N_{\sigma}$ has Lebesgue measure $0$.\\
The equation (\ref{eqb}) can be written as:
\begin{eqnarray*}
X_{t}&=&X_{0}+\int_{0}^{t}\sigma(X_{s})\;dB_{s}+\int_{0}^{t}b(X_{s})1_{N_{\sigma}}(X_{s})\;ds+\int_{0}^{t}\frac{b(X_{s})}{\sigma^{2}(X_{s})}.\sigma^{2}(X_{s})1_{N_{\sigma}^{c}}(X_{s})\;ds
\end{eqnarray*}
By $(LT)$ condition
$$\int_{0}^{t}b(X_{s})1_{N_{\sigma}}(X_{s})\;ds=0,$$
then using occupation times formula, we may reexpress $X_{t}$ as follows :
\begin{eqnarray*}
X_{t}&=&X_{0}+\int_{0}^{t}\sigma(X_{s})\;dB_{s}+\int_{0}^{t}\frac{b(X_{s})}{\sigma^{2}(X_{s})}1_{N_{\sigma}^{c}}(X_{s})\;d\langle X\rangle_{s}\\
&=& X_{0}+\int_{0}^{t}\sigma(X_{s})\;dB_{s}+\int_{\mathbb{R}}L_{t}^{x}(X)\mu(dx)\\
\end{eqnarray*}
with $L_{t}^{x}(X)$ is the local time of $X$, $\mu(dx)=\frac{b(x)}{\sigma^{2}(x)}1_{N_{\sigma}^{c}}(x)\;dx$.
 Since \\$\frac{b}{\sigma^{2}}1_{N_{\sigma}^{c}}\in L_{loc}^{1}(\mathbb{R})$, $\mu$ is a $\sigma$-finite signed measure. By Theorem  \ref{EDSb}, there exists a unique pathwise solution to (\ref{eqb}).
\end{proof}
\begin{Example}
If $N_{\sigma} \subset N_{b}$ or $N_{\sigma}$ is at most countable, then $\int_{0}^{t}1_{N_{\sigma}}(X_{s})\;ds=0$ for all $t >0$.
\end{Example}
\subsection{On One SDE with non-sticky boundary conditions}\label{On One SDE with non-sticky boundary conditions}
\begin{thm}\label{Th5}
Consider the following equation
\begin{eqnarray}\label{EDSa}
X_{t}&=&X_{0}+ \int_{0}^{t}a(X_{s})\;dB_{s}+ \int_{0}^{t} b(X_{s})\;ds.
\end{eqnarray}
Suppose that $a(x)$,  $b(x)$ are bounded continuous functions which satisfy the following conditions;
\begin{enumerate}
  \item $a$ satisfies $(LT)$ condition,
  \item $a(0)=0$, $a(x)\neq 0$ for $x \neq 0$ and $b(0) \neq 0$.
\end{enumerate}
Then the pathwise uniqueness holds for(\ref{EDSa}).
\end{thm}
In order to lighten the proof, we begin by the following Lemma:
\begin{lem}\label{origin}
Under the assumptions of Theorem \ref{Th5}, any solution of the equation (\ref{EDSa}) satisfies
$$\int_{0}^{t}1_{\{X_{s}=0\}}\;ds=0, \;\;\; \forall t \geq 0,\;\;\; a.s.$$
\end{lem}
\begin{proof}
Let $X$ be a solution of (\ref{EDSa}).
We have

\begin{eqnarray*}
L_{t}^{0}(X)- L_{t}^{0^{-}}(X)&=&2 \int_{0}^{t}1_{\{X_{s}=0\}}\;dX_{s}\\
&=& 2 \int_{0}^{t}1_{\{X_{s}=0\}}a(X_{s})\;dB_{s}+ 2\int_{0}^{t}1_{\{X_{s}=0\}}b(X_{s})\;ds.
\end{eqnarray*}
$(LT)$ condition implies that $L_{t}^{0}(X)- L_{t}^{0^{-}}(X)=0$.
Since $a(0)=0$ and $b(0)\neq 0$,
we must have
$$\int_{0}^{t}1_{\{X_{s}=0\}}\;ds=0$$
\end{proof}
\begin{lem}\label{compar}
Consider the following two equations;
\begin{eqnarray}\label{EDScomp}
X_{t}^{i}=X_{0}^{i}+\int_{0}^{t}a(X_{s}^{i})\;dB_{s}+\int_{0}^{t}b^{i}(X_{s}^{i})\;ds, \;\;\; i=1,\;2,
\end{eqnarray}
where $a(x)$, $b^{i}(x)$ are bounded continuous functions.
\begin{enumerate}
  \item [(i)]$a$ satisfies $(LT)$ condition,
  \item[(ii)] $b^{1}(x)<b^{2}(x)$ for any $x\in \mathbb{R}$.
\end{enumerate}
If $X^{1}$ and $X^{2}$ are solutions of (\ref{EDScomp}) and $X_{0}^{1}\leq X_{0}^{2}$, then $X_{t}^{1}\leq X_{t}^{2}$ for $\forall t \geq 0$ with probability one.
\end{lem}
\begin{proof}
Let $X^{1}$ and $X^{2}$ be solutions of (\ref{EDScomp}) corresponding to $b^{1}$, $b^{2}$ respectively.\\
In case $(i)$ we can Suppose that either $b^{1}$ or $b^{2}$ are Lipschitz continuous, because it is possible to find a Lipschitz function $l$ such that $b^{1}<l< b^{2}$.
By Tanaka's formula:
\begin{eqnarray*}
(X_{t}^{1}-X_{t}^{2})^{+}&=&(X_{0}^{1}-X_{0}^{2})^{+}+\int_{0}^{t} 1_{\{X_{s}^{1}>X_{s}^{2}\}}(a(X_{t}^{1})-a(X_{t}^{2}))\;dB_{s}\\
&+& \int_{0}^{t}1_{\{X_{s}^{1}>X_{s}^{2}\}}(b^{1}(X_{s}^{1})-b^{2}(X_{s}^{2}))\;ds+\frac{1}{2}L_{t}^{0}(X^{1}-X^{2}).
\end{eqnarray*}
The function $a$ satisfies $(LT)$ condition, therefore,
$$\psi(t)=\mathbb{E}[(X_{t}^{1}-X_{t}^{2})^{+}] \leq \mathbb{E}\left[\int_{0}^{t} 1_{\{X_{s}^{1}>X_{s}^{2}\}}(b^{1}(X_{s}^{1})-b^{1}(X_{s}^{2}))\;ds\right].$$
Thus, if $b^{1}$ is Lipschitz with constant $C_{t}$,
$$\psi(t) \leq C_{t}. \mathbb{E}\left[\int_{0}^{t}1_{\{X_{s}^{1}>X_{s}^{2}\}}|X_{s}^{1}-X_{s}^{2}|\;ds\right]=C_{t}\int_{0}^{t}\psi(s)\;ds,$$
and we conclude by using Gronwall's lemma and the usual continuity arguments.\\
If $b^{2}$ is Lipschitz, using the same arguments, we get:
$$\psi(t) \leq  \mathbb{E}\left[\int_{0}^{t}1_{\{X_{s}^{1}>X_{s}^{2}\}}|b^{2}(X_{s}^{1})-b^{2}(X_{s}^{2})|\;ds\right].$$
Since $b^{1}\leq b^{2}$,  we complete the proof as in the first case.
\end{proof}
\\

Now, we are ready to give the proof of Theorem.
\\
\begin{proof}[Proof of Theorem]
The proof consists in the construction of the minimal and the maximal solutions $\underline{X}$ and $\overline{X}$ respectively. This construction is due to Y. Ouknine \cite{OUrea}. Under the assumptions of Theorem \ref{Th5}, $b$ is continuous, then there exists a sequence $b_{n}$ Lipschitz nondecreasing in $n$, and such that if $x_{n}$ converges to $x$, $b_{n}(x_{n})$ converges to $b(x)$. Let $X_{n}$ be the unique strong solution to the SDE:
\begin{eqnarray}
X_{t}&=&X_{0}+ \int_{0}^{t}a(X_{s})\;dB_{s}+ \int_{0}^{t} b_{n}(X_{s})\;ds.
\end{eqnarray}
By comparison Theorem, the process $X^{n}=\{X_{t}^{n},\;\;t \geq 0\}$ is increasing in $n$, a.s. For $t\geq 0$, we define $\underline{X}_{t}=\lim_{n\rightarrow+\infty}X_{t}^{n}$. \\
We denotes $X^{\pm K}$ the solutions of the SDE's:
$$X_{t}=x+\int_{0}^{t}\sigma(X_{s})\;dB_{s}\pm (K+1) t$$
We can assume that $b_{n}(x)\leq K+1,\;\;\forall n \in \mathbb{N}$. Thus, $|\underline{X}_{t}|\leq|X_{t}^{-K}|+|X_{t}^{+K}|$, by sample estimations, we deduce that $\mathbb{E}\sup_{s \leq t}|X_{s}^{\pm K}|^{2} < \infty$. Therefore, $|\underline{X}_{t}|< \infty$ a.s., If we tend $n \rightarrow \infty$, we prove that $|\underline{X}|$ is a strong solution to the SDE:
\begin{eqnarray}
X_{t}&=&X_{0}+ \int_{0}^{t}a(X_{s})\;dB_{s}+ \int_{0}^{t} b(X_{s})\;ds.
\end{eqnarray}
The construction of $\overline{X}$ being treated in similar fashion.
It is clear that $\underline{X}$ and $\overline{X}$ have the strong Markov property. So $\overline{X}$ and $\underline{X}$ are  diffusions processes with the same local generator
$$L=a^{2}(x)\frac{d^{2}}{dx^{2}}+b(x)\frac{d}{dx}$$
and have no stay at the origin by Lemma \ref{origin}.
However,  by Feller's general theory of one dimensional diffusion processes, there exists only one diffusion process which possesses $L$ as its local generator and does not stay at the origin. Hence the pathwise uniqueness.
\end{proof}
\section{A generalization of the perturbed Tanaka equation }\label{Tanaka}
Prokaj in \cite{prokaj} has showed a recent result on the pathwise uniqueness of the so-called perturbed Tanaka equation:
\begin{eqnarray}\label{ta}
Y_t= y+\int_0^t sign(Y_s)dM_t+  N_t
\end{eqnarray}
\begin{thm}[Prokaj 2010]
Suppose that $M$,  $N$ are continuous local martingales with $M_0=N_0=0$ and quadratic and cross-variations that satisfy the condition of orthogonality and domination
$$\langle M, N\rangle_t=0\;\;\;\;\langle M\rangle_t =\int_0^t q(s)\langle N\rangle_s\;\;\;\; s \leq  t$$
respectively, for some progressively measurable process $q(.)$ with values in a compact  interval $[0,c]$. Under these assumptions pathwise uniqueness holds for the perturbed Tanaka equation (\ref{ta}).
\end{thm}
In this section, we use the local time technics introduced by Perkins \cite{perkins} and further developed by LeGall \cite{legall2}, to provide a simple proof of a more general result of this type.\\
Let $(\Omega,(\F_t)\tgo,\P)$ be a filtered probability space and
$B=(B^{(1)},B^{(2)})$ be a two dimensional Brownian motion
in the filtration
$(\F_t)\tgo$.
We are interested in the uniqueness of the solution for
the following equation
\begin{equation}
  \label{eq:XU}
  dX_t=\sigma(X_t)dB^{(1)}_t+\lambda dB^{(2)}_t,
\end{equation}
where $\lambda\in \mathbb{R}$ is a constant and $\sigma: \mathbb{R}\rightarrow \mathbb{R}$ is a measurable function which satisfies the assumption ($2$BV) below.
\begin{thm}\label{thm:1}
 Suppose that there exists a function $f$ of bounded variation such that for every real numbers $x$, $y$
 $$|\sigma(x)- \sigma(y)|^2 \leq  |f(x)-f(y)|\;\;\;(\mbox{2BV})$$
 if $\lambda\neq0$, then the solution of \eqref{eq:XU} is pathwise unique.

\end{thm}
Inspired by Prokaj \cite{prokaj} and \cite{kara}, we prove a more general statement than Theorem \ref{thm:1}.
Let  $M,N$ be tow local martingales, we say that $M$ and $N$ are strongly orthogonal martingale if $\langle M,N\rangle=0$ i.e. whose product is a local martingale.\\ We say
that {\em $N$ dominates $M$} if for some constant $c>0$ we have $\langle M\rangle\leq
c\langle N\rangle$. In other words there is a process $Q$ (it can be chosen to be
predictable) such that $\langle M\rangle_t=\int_0^t Q_s d\langle N\rangle_s$ for all $t\geq 0$
and $\mathbb{P}(\forall s\geq 0,\,0\leq Q_s\leq c)=1$. A localized version of this
notion,  namely {\em $N$ locally dominates $M$}, holds if this $Q$ is locally
bounded.

\begin{thm}\label{thm:2}
  Let  $M,N$ be continuous local martingales in $(\F_t)\tgo$. Assume that $M$
  and $N$ are strongly orthogonal 
  and $N$ dominates $M$.
  Suppose that there exists a function $f$ of bounded variation such that for every real numbers $x$, $y$
 $$|\sigma(x)- \sigma(y)|^2 \leq  |f(x)-f(y)|,\;\;\;(\mbox{2BV})$$
  then the solution of the
  equation
  \begin{equation}
    \label{eq:XMN}
    dX_t=\sigma(X_t)dM_t+dN_t
  \end{equation}
  is pathwise unique.
\end{thm}
\begin{proof}
Without loss of generality, we shall prove the statement for an increasing function $f$.
Let us first show that $L_{.}^{0}(X-Y)\equiv0$, whenever $X$ and $Y$ denote any two solutions of the SDE (\ref{eq:XMN}) with
the same underlying local martingales $M$ and $N$.
By the right continuity of $L_{.}^{0}$ it is enough to prove that, for any $t\geq 0$,
\begin{equation*}
\int_{0^+}^{+\infty }\frac{L_{t}^{a}\left( X-Y\right)}{a}da<+\infty.
\end{equation*}
Indeed, using the density occupation formula we can write for any $\delta >0$,
\begin{equation*}
\int_{0^+}^{+\infty }\frac{L_{t}^{a}(X-Y)}{a}da=\int_{0}^{t}\frac{d\langle
X-Y\rangle_{s}}{X_{s}-Y_{s}}1_{\left\{ X_{s}-Y_{s}>0 \right\}
}=\int_{0}^{t}\frac{\left( \sigma (X_{s})-\sigma (Y_{s})\right)
^{2}}{%
X_{s}-Y_{s}}1_{\left\{ X_{s}-Y_{s}>0 \right\} }d\langle M\rangle_s.
\end{equation*}
Applying the assumption (BV2) we obtain
\begin{equation*}
\int_{0}^{t}\frac{\left( \sigma (X_{s})-\sigma (Y_{s})\right)
^{2}}{%
X_{s}-Y_{s}}1_{\left\{ X_{s}-Y_{s}>0\right\} }d\langle M\rangle_s\leq
\int_{0}^{t}\frac{%
|f(X_{s})-f(Y_{s})|}{X_{s}-Y_{s}}1_{\left\{ X_{s}-Y_{s}>0
\right\} }d\langle M\rangle_s.
\end{equation*}
As a consequence,
\begin{eqnarray}
 \E\left[\int_{0^+ }^{+\infty }\frac{L_{t}^{a}(X-Y)}{a}da\right]\leq& \E\left[\int_{0}^{t}\frac{
|f(X_{s})-f(Y_{s})|}{X_{s}-Y_{s}}1_{\left\{ X_{s}-Y_{s}>0
\right\} }d\langle M\rangle_s\right].
\end{eqnarray}
Now, by a localization argument $\|f\|_{\infty}:=sup_{x}|f(x)|< \infty$.\\
Let $\theta _{n}$ denote the standard positive regularizing mollifiers sequence, and define
$$f_{n}(x)=(\widetilde{f}(.)*\theta _{n})(x)\;\;\;\;\;\;\mbox{for}\;\; x\in \mathbb{R},\;\;\;n \in \mathbb{N}^*,$$
where $\widetilde{f}$ is any real function such that $\widetilde{f}(x)=f(x)$ if $|x|\leq M$ and $0$ \\if $|x| \geq M+1$.
\\
Note that $f_{n}$ are increasing functions, with support contained in\\ $[-M-1,M+1]$ such that
$$\sup_m\sup_x|f_n(x)|\leq \|f\|_{\infty}\;\;\;\;\;\; \mbox{and};\;\;\;\;f_{n}(x)\rightarrow f(x)\;\;\; \mbox{for every}\;\; x \in D^c,\;\; |x|\leq M$$
where $D$ is the denumerable set of discontinuous points of the function $f$.
If we denote $Z_t^\alpha =\alpha X_t+(1-\alpha)Y_t $
 ,the orthogonality of $M$ and $N$ implies that $\langle N\rangle(.)\leq  \langle Z^\alpha\rangle(.)$, Let us recall the domination assumption which gives
\begin{eqnarray}\label{domination}
\langle M\rangle(.) \leq c \langle N\rangle(.) \leq c\langle Z^\alpha\rangle(.).
\end{eqnarray}
Hence, using successively Fatou's Lemma, the intermediate value theorem, and the domination assumption (\ref{domination}) we get,
\begin{align*}
& \E\left[\int_{0}^{t}\dfrac{\left( f(X_{s})-f(Y_{s})\right)
}{X_{s}-Y_{s}} 1_{\left\{ X_{s}-Y_{s}>0 \right\} }d\langle M\rangle_s\right]\\
&\leq \liminf_{n\rightarrow +\infty }\E\left[\int_0^t\dfrac{\left(
f_{n}(X_{s})-f_{n}(Y_{s})\right) }{X_{s}-Y_{s}}1_{\left\{
X_{s}-Y_{s}>0 \right\} }d\langle M\rangle_s\right]\\
&= \liminf_{n\rightarrow +\infty
}\E\left[\int_{0}^{t}\int_{0}^{1}\dfrac{%
\partial f_{n}}{\partial a}(\alpha X_{s}+(1-\alpha )Y_{s})d\alpha d\langle M \rangle_s \right]\\
&= \liminf_{n\rightarrow +\infty
}\int_{0}^{1}d\alpha\E\left[\int_{0}^{t}\dfrac{
\partial f_{n}}{\partial a}(Z_s^\alpha)d\langle M\rangle_s\right]\\
& \leq c\liminf_{n\rightarrow +\infty
}\int_{0}^{1}d\alpha\E\left[\int_{0}^{t}\dfrac{
\partial f_{n}}{\partial a}(Z_s^\alpha)d\langle Z^\alpha\rangle_s\right]
\end{align*}
Note that we have used in the first inequality the fact that
\begin{eqnarray}\label{statement}
 \int_{0}^{t}P[ (X_{s}\in D)\cup (Y_{s}\in D)]d \langle M \rangle_s &=& 0.
\end{eqnarray}
To see that this statement holds, it suffices to remark that

\begin{eqnarray*}
\int_{0}^t 1_{\{X_s=a\}}d \langle M \rangle_s \leq c \int_{0}^t 1_{\{X_s=a\}}d \langle N \rangle_s\leq \int_{0}^t 1_{\{X_s=a\}}d \langle X \rangle_s=0\;\;\;\; \forall a \in \mathbb{R}
\end{eqnarray*}
The first inequality due to the domination.
The last equality is a consequence of occupation times formula. Hence,
\begin{eqnarray}
 \E\left[\int_{0}^{+\infty }\frac{L_{t}^{a}(X-Y)}{a}da\right] &\leq& \liminf_{n\rightarrow +\infty}\E\left[\int_{\mathbb{R}} \int_{0}^{1}\dfrac{\partial f_{n}}{\partial
a}(a)L_{t}^{a}(Z^\alpha) d\alpha da\right].
\end{eqnarray}
However, since $\alpha \in (0,1)$, then standard
calculations imply that, for any $p\geq 0$ and $t\geq 0$,
\begin{equation}
\label{esti1} \E[\sup_{s\leq t}|X_s|^p]<\infty.
\end{equation}
 Using Tanaka formula and the inequality
$|Z_t^\alpha-a|-|Z_0^\alpha-a|\leq
|Z_t^\alpha-Z_0^\alpha|$  we deduce that
$$
\sup_{\alpha \in [0,1],a\in R}\E \left[L_{t}^{a} (Z^\alpha)\right]<\infty.$$ Therefore, we obtain
\begin{eqnarray*}
 \E\left[\int_{0^+ }^{+\infty }\frac{L_{t}^{a}(X-Y)}{a}da\right] &\leq& c \sup_{\alpha
\in [0,1],a\in R}\E\left[ L_{t}^{a}
(Z^\alpha)\right]\int_R \dfrac{\partial f_{n}%
}{\partial a}(t,a)da \\
& \leq & C\left\| f\right\|_{\infty }
\end{eqnarray*}
where \ $C>0$\ is a generic constant. hence $L_{.}^{0}(X-Y)\equiv0$, by Tanaka formula, we obtain that $|X_.-Y_.|$ is a local martingale, thus  also a nonnegative supermartingale, with $|X_0-Y_0|=0$, and consequently, $X$ and $Y$ are indistinguishable.
\end{proof}
\begin{rem}
In \cite{kara}, the authors  have showed theorem but under the following elementary comparison:
$$|\sigma(x)-\sigma(y)|^2 \leq  \|\sigma\|_{TV}|\sigma(x)-\sigma(y)| $$
where $  \|\sigma\|_{TV}$ is the total variation of $\sigma$. As we have seen previously, this comparison can be substituted by a weaker assumption ( 2 BV).
\end{rem}
\begin{rem}
We think that the power $2$ in the assumption (2 BV) is sharp as in \cite{Barlow2}.
\end{rem}
The equation (\ref{eq:XU}) with a suitably correlated Brownian motion with high variance can restore pathwise uniqueness. The following result can be seen also as a generalization of this equation
$$dX_t=\lambda dt+1_{\{X_t> 0\}}, dW_t+(\eta/2)dV_t,\;\;\;\;\;0 \leq  t < \infty$$
considered in \cite{kar}.
\begin{thm}
Let $A$ be a process of bounded variation, and $W$, $V$ are standard Brownian motions, the following equation
\begin{eqnarray}\label{stren}
dX_t=\sigma(X_s)dW_s+(\eta/2)dV_t+dA_t
\end{eqnarray}
has a pathwise unique solution, provided either
\begin{itemize}
\item[(i)]$\eta \neq [-1,1]$ and $\langle W,V\rangle_t=(-t/ \eta)$, $0 \leq  t < \infty$, or
\item[(i)]$\eta \neq 0$ and $W$ and $V$ are independent.
\end{itemize}
\end{thm}
\begin{proof}
The equation (\ref{stren}) is equivalent to the following:
$$dX_t=(2\sigma(X_s)-1)dM_s+N_t+A_t$$
where the process $M:= W/2$, $N:=(W+\eta V)/2$ are continuous, orthogonal martingales with quadratic variation $<M>_t=t/4$ and
\\
 $<N>_t=(\eta^2-1)t/4$, respectively.

\end{proof}
\subsection{Some open problems}\label{Open}
Let us mention an open problems which we think are quite interesting.\\
\textbf{Problem1}:
\\

We consider the  following stochastic equation
\begin{eqnarray}\label{Le gall}
X_{t}=X_{0}+B_{t}+\int_{\mathbb{R}}L_{t}^{a}(X)\;\nu(da)
\end{eqnarray}
where $L_{t}^{a}(X)$ denotes the local time at $a$ for the time $t$ for the semimartingale $X$, $\nu$ is bounded measure on $\mathbb{R}$.
Following Le Gall \cite{Legall}, the solution of (\ref{Le gall})is obtained as the limit of a sequence of solutions satisfying:
$$X_{t}^{n}=X_{0}+B_{t}+\int_{0}^{t}b_{n}(X_{s}^{n})\;ds$$
when the measure $\mu^{n}(da)=b_{n}(a)\;da$ converges weakly to some measure $\mu$. The measure $\nu$ in (\ref{Le gall}) is obtained from the limit $f$ of $f_{\mu_{n}}(x)=\exp(-2 \int_{0}^{x}b_{n}(a)\;da)$ by this formula:
$$\nu(da)=\frac{f^{'}(da)}{f(a)+f(a^{-})},$$
with $f(a^{-})$ denotes the left limit of $f$ at a point $a$.\\
 What one can say for the non homogenous case?
 $$X_{t}^{n}=X_{0}+B_{t}+\int_{0}^{t}b_{n}(s,X_{s}^{n})\;ds$$
 which is equivalent to the sequence of this SDE:
 $$X_{t}^{n}=X_{0}+B_{t}+\int_{0}^{t}\int_{\mathbb{R}}b_{n}(s,a)\,d_{s}L_{s}^{a}(X^{n})\;da.$$
 Formally, this sequence will converge to :
 \begin{eqnarray}\label{general}
 X_{t}=X_{0}+B_{t}+\int_{0}^{t}\;d_{s}L_{s}^{\nu_{s}}(X)
 \end{eqnarray}
where  $L_{t}^{\nu_{s}}=\int_{\mathbb{R}}L_{s}^{a}(X)\;\nu_{s}(da)$.\\
The pathwise uniqueness of the solution to equation (\ref{general}) was studied by  S. Weinryb \cite{Weinryb} in the case $\nu_{s}(da)=\alpha(s)\delta_{0}(da) $ where $\alpha:\mathbb{R}^{+}\rightarrow \mathbb{R}$ is a deterministic function and $\alpha \leq 1/2$.\\
In the other hand, we have shown the pathwise uniqueness for (\ref{general}) when  the function $\alpha$ is constant. However, we don't know whether the pathwise uniqueness hold in the other cases.
\\
\textbf{Problem2}:
In \cite{buc}, the authors has studied the weak limit of the following equation:
$$dX_t^\epsilon= b(X_s^\epsilon)ds +\epsilon dB_t,$$
one can ask, what can be the limit of the following equation :
$$dX_t^\epsilon= \sigma(X_s^\epsilon)dM_t +\epsilon dN_t$$
where $M$ and $N$ are two local martingales which satisfy the same assumptions as in section \ref{Tanaka}.
\end{spacing}


\newpage
\begin{thebibliography}{1}
\bibitem{Yor} J. Azema, M. Yor. En guise d'introduction. Temps locaux, Astérisque n° $52-53$, soc. Math. France $1978$.
\bibitem{Barlow} M. T. Barlow et E. Perkins. SDE's singular increasing process, Stochastics, $12$ $(1984)$, $229-242$.
\bibitem{Barlow2} M.T. Barlow, One dimensional stochastic differential equations with no strong solution. J. London math. Soc.(2) 26, 335-347.
\bibitem{belfa} R. Belfadli et Y. Ouknine. Unicité trajectorielle des éuations différentielles stochastiques avec temps local et temps de séjour au bord. Afrika matematika AM 21(1) 2011.
\bibitem{benabdallah} M. Benabdallah, S. Bouhadou and Y. Ouknine. On the pathwise uniqueness of solutions of one-dimensional stochastic differential equations with jumps. Preprint.
\bibitem{buc} R. Buckdahn, Y. Ouknine and M. Quincampoix, (2009), on limiting values of stochastic differential equations with small noise intensity tending to zero, Bull, Sci, Math 133: 229-237.
\bibitem{coquet} F. Coquet, Y. Ouknine, "Some identities on semimartingales local times." Statist. and Probab. Letters $49$, $149-153$.
\bibitem{ENgelbert-Yamada} H. J. Engelbert. On the theorem of T. Yamada and S. Watanabe. Stochastics Rep. 36, 3-4, 205-216 (1991). MR1128494.
 \bibitem{Schmidt} H. J. Engelbert and W. Schmidt, One-dimensional stochastic differential equations with generalized drift. Stochastic differential systems (Marseille-Luminy, 1984), 143-155, Lecture notes in
     control and information sciences 69, Springer-verlag, Berlin 1985.
\bibitem{kara} E. R. Fernholz, T. Ichiba, I. Karatzas and V. Prokaj. Planar diffusions with rank-base characteristics and perturbed Tanaka equation. (2011)
  \bibitem{Harisson}J. M. Harisson and L. A. Shepp. On skew Brownian motion, Ann. probab. 9, 309-313 (1981).
   \bibitem{kar} I. Karatzas, A, N. Shiryaev and M. Shkolnikov. On the one-sided Tanaka equation with drift; arxiv.
 \bibitem{zvonkin} N.V. Krylov, A.K. Zvonkin. On strong solutions of stochastic differential equations.  Selecta Mathematica Sovietica, 1981, vol. 1, no. 1, 1961.
\bibitem{legall2}J. F. LeGall.Applications des temps locaux aux équations différentieles stochastiques unidimensionelles. Lectures notes in Mathematics $986$. $15-31$ Springer-Verlag. New York (1983)
\bibitem{Legall}J. F. LeGall. One-dimensional stochastic differential equations involving local times of unknown process. Stochastic analysis and application (Swansea 1983) 51-82. Lecture notes in mathematics 1095, Springer-verlag, Berlin 1984.
\bibitem{Shiga} S. Manabe and  T. Shiga, $(1973)$. On one-dimensional stochastic differential equations
with non-sticky boundary condition. J. Math. Kyoto Univ. $13$ $595-603$.
\bibitem{Nakao} S. Nakao $(1972)$. On the pathwise uniqueness of solutions of one-dimensional
stochastic differential equations. Osaka J. Math. $9$ $513-518$.
\bibitem{Nasyrov}S. F. Nasyrov, On local times for functions and stochastic processes: II. Theory Probab.
Appl. 41 (1996), $275-287$.
\bibitem{Sup} Y. Ouknine, Temps Local du Produit et du Sup de deux semimartingales, Séminaire de
probabilités de Strasbourg, Volume $24$, pp. $477-479$ $(1989)$
\bibitem{Ouk Nakao} Y. Ouknine, Généralisation d'un lemme de S. Nakao et applications, Stochastics $23$ $(1988)$ $149-158$.
\bibitem{OUrea} Y. Ouknine, Sur la réalisation des solutions déquations différentielles stochastiques. Afr. Mat., Sér. II  $71-76$ $(1993)$.
\bibitem{Ident} Y. Ouknine, Quelques identités sur les temps locaux et unicité des solutions d{'}équations
différentielles stochastiques avec reflection, Stochastic Processes and their Applications, Volume
$48$, Issue $2$, pp. $335-340$ $(1993)$
\bibitem{Rutkowski} Y. Ouknine et M. Rutkowski, Local times of functions of continuous Semimartingales,
Stochastic Anal. Appl. $13$, no. $2$, pp. $211-231$ $(1995)$
\bibitem {Ouknine} Y. Ouknine : Unicité trajectorielle	des equations différentielles stochastiques avec temps local, Probab. and Math. Stat., \textbf{19}, Fasc. 1 (1999), 55-62.
\bibitem{perkins} Local time and pathwise uniqueness for stochastic differential equations. Lecture notes in Mathematics $920$. 201-208. Spriner-Verlag, New york. (1982).
\bibitem{prokaj} V. Prokaj, The solution of the perturbed Tanaka equation is pathwise unique, Preprint.(2010)
\bibitem {Protter} Protter, P.: Stochastic integration and differential equations. 2nd ed. Springer 2005.
\bibitem{Revuz} D. Revuz and  M. Yor, $(1999)$. Continuous Martingales and Brownian Motion.
Springer, Berlin.
\bibitem{M.Rutkowski} M. Rutkowski, Strong solutions of stochastic differenetial equations involving local times. Stochastics an international journal of
probability and stochastic processes, 22:3. 201-218.
\bibitem{yamada} T.Yamada and S. Watanabe. On the uniqueness of solutions of stochastic differential equations. Kyoto
\bibitem{Weinryb} S. Weinryb, Etude dune équation différentielle stochastique avec temps local. In Séminaire
de Probabilités XVII. Lecture Notes in Math., Volume $986$, Springer-Verlag, pp. $72-77$ $(1982)$

\end{thebibliography}
\end{document}